\documentclass{article}
\usepackage{amsmath, amsthm, amssymb, amsfonts, multirow, tikz-cd, pgfplots, xcolor, mdframed}
\pgfplotsset{width=200pt, compat=1.9}
\usepackage[english]{babel}
\usepackage[utf8]{inputenc}
\usepackage[margin=1in]{geometry}
\usepackage{hyperref, framed}

\newtheorem{question}{Question}[section]

\newtheorem{theorem}[question]{Theorem}
\newtheorem{lemma}[question]{Lemma}

\newtheorem{proposition}[question]{Proposition}

\newcommand{\ignore}[1]{}
\title{\vspace*{-.5in}Maximum distances in the four-digit {K}aprekar process}
\author{Pat Devlin\footnote{Yale University, New Haven CT, USA \qquad \texttt{patrick.devlin@yale.edu}} \and Tony Zeng\footnote{Yale University, New Haven CT, USA \qquad \texttt{tony.zeng@yale.edu}}}

\date{October 21, 2020}

\begin{document}

\maketitle
\renewcommand{\thefootnote}{\fnsymbol{footnote}}
\footnotetext{AMS 2020 subject classification: 05A15, 00A08, 11B37, 11B75}
\footnotetext{Key words and phrases:  Kaprekar routine, Kaprekar function, recreational mathematics, base-dependent integer problems}

\setcounter{footnote}{0}
\renewcommand{\thefootnote}{\arabic{footnote}}

\begin{abstract}
For natural numbers $x$ and $b$, the classical Kaprekar function is defined as $K_{b} (x) = D-A$, where $D$ is the rearrangement of the base-$b$ digits of $x$ in descending order and $A$ is ascending.  The bases $b$ for which $K_b$ has a $4$-digit non-zero fixed point were classified by Hasse and Prichett, and for each base this fixed point is known to be unique.  In this article, we determine the maximum number of iterations required to reach this fixed point among all four-digit base-$b$ inputs, thus answering a question of Yamagami.  Moreover, we also explore---as a function of $b$---the fraction of four-digit inputs for which iterating $K_b$ converges to this fixed point.
\end{abstract}

\section{Introduction}
In 1949, Dattatreya Ramchandra Kaprekar \cite{kaprekar1949another} introduced the following process.  We start with a base-$b$ four-digit number $x$ (allowing this to contain leading zeros such as $x=0309$).  Rearrange the digits of $x$ to be in decreasing order and subtract from this the rearrangement of the digits written in increasing order (these operations being done on base-$b$ integer representations).  This yields another $4$-digit base-$b$ number, and we denote this output as $K_b (x)$.  For instance, in base $10$, we would have
\begin{eqnarray*}
K_{10} (3223) &=& 3322 - 2233 = 889\\
K_{10} (0889) &=& 9880 - 0889 = 8991\\
K_{10} (8991) &=& 9981 - 1899 = 8082\\
K_{10} (8082) &=& 8820 - 0288 = 8532\\
K_{10} (8532) &=& 8532 - 2358 = 6174\\
K_{10} (6174) &=& 7641 - 1467 = 6174.
\end{eqnarray*}
From the above, we see 6174 is a fixed point of $K_{10}$.  Exploring this further, Kaprekar discovered that in fact, if we take any $4$-digit base-$10$ integer not divisible by $1111$, then iterating $K_{10}$ starting at that input will necessarily reach 6174 within at most $7$ steps.  (Multiples of $1111$ are sent to $0$, which is a trivial fixed point.)  This result was subsequently popularized by Martin Gardner, who featured it in his March 1975 column of \textit{Mathematical Games} \cite{sciAm} published in Scientific American.\footnote{We also recommend the first puzzle presented in this column, concerning a worm traversing an ever-stretching rubber rope.}

Although the study of this map began as recreational exploration, it has since been shown that the $2$-digit version is closely related to Mersenne primes \cite{yamagami20182}, and quite recently this procedure has found independent applications to cryptographic encoding schemes \cite{nandan2020multi}.  Several variations of the Kaprekar process have been studied as well \cite{young1993variation, young1995switch, chaillekaprekar}, and we point to \cite{yamagami20182} and \cite{peterson2008kaprekar} for convenient overviews of the literature.

Many authors seeking to understand the behavior of this problem turned to searching for fixed points in various bases \cite{trigg1972kaprekar, lapenta1979algorithm, walden2005searching}, and a lot has been published on the subject.  To this end, Ludington \cite{ludington1979bound} studied the behavior of $K_b$ with $b$ remaining fixed and a variable numer of digits, and she proved that for each fixed $b$, there are only finitely many $r$ for which almost every $r$-digit base-$b$ number is sent to the same fixed point.  Similarly, the base-10 fixed points have been carefully studied in \cite{prichett1981determination, dolan2011classification}.  Yamagami and Matsui \cite{yamagami2019some} recently explored this for other bases as well, proving a lower bound for the number of base-$b$ fixed points in terms of the number of non-trivial divisors of $b$.

A second approach has been to fix the number of digits and analyze the corresponding map for different bases.  Among the literature most relevant for us, Hasse and Prichett \cite{hasse1978determination} studied the four-digit version of this map to analyze for which bases there exists a fixed point of $K_b$, and they also characterized the bases for which almost every starting point converges to such a fixed point.  Prichett \cite{prichett1978terminating} subsequently obtained similar results for $5$-digit numbers.

For $3$-digit numbers, Eldridge and Sagong \cite{eldridge1988determination} proved that a base-$b$ fixed point exists iff $b$ is even.  In this case, they showed that every input not divisible by $111$ is eventually mapped to this fixed point.  For $b$ odd, they proved that repeated iteration of $K$ for nearly all $x$ results in a loop of period 2.  Their analysis also determines the maximum number of iterations needed to reach this 3-digit fixed point.

Finally, Yamagami \cite{yamagami20182} discussed the behavior of $K_b$ for a range of examples but with particular focus on the $2$-digit case.  Yamagami also asked the question: \textit{in the four-digit case, what is the furthest finite distance a point is away from a fixed point?}

In other words, let $S_b$ denote the set of $4$-digit base-$b$ inputs for which $K^{t}_{b} (x)$ is eventually a non-zero constant (as $t$ increases).  We then define $M_b$ to be the least $m$ for which $K^{m}_b (x) = K^{m+1}_b (x)$ for all $x \in S_b$.

\begin{figure}[h]
	\begin{center}
		\includegraphics[height=1.5in]{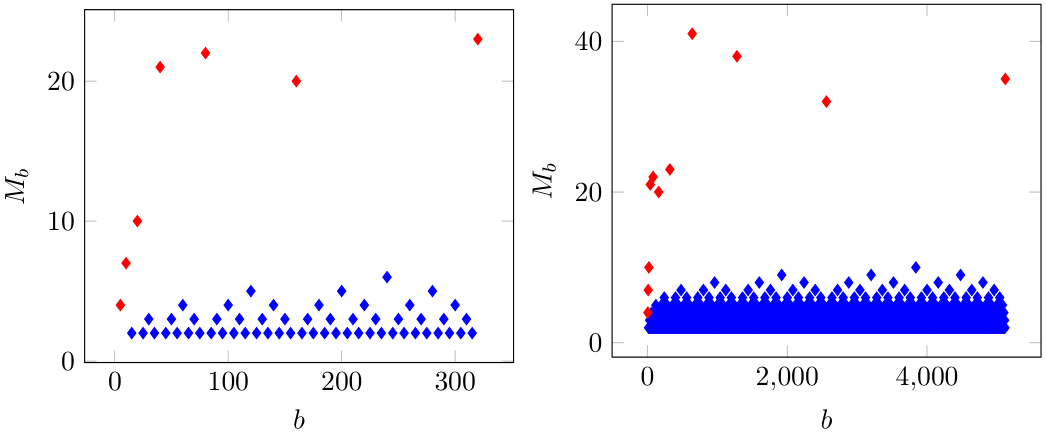}
		\vspace*{12pt}
		\includegraphics[height=1.5in]{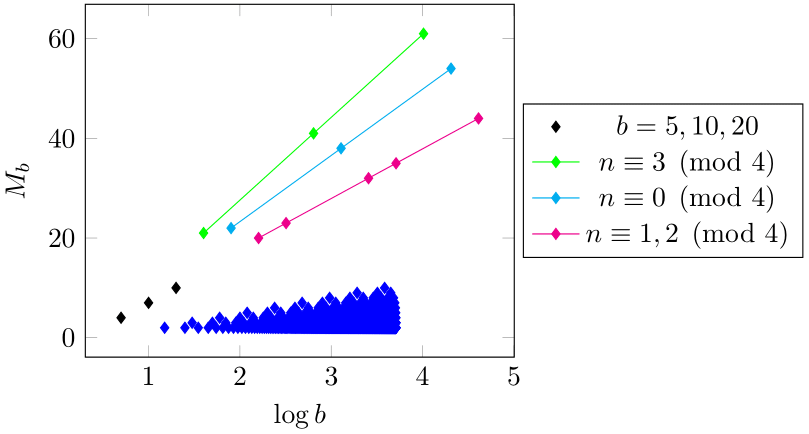}
	\caption{Plots of (computed) values of $M_b$ shown at different scales.} \label{figure:mb}
	\end{center}
\end{figure}

Previous literature has studied $M_b$ in the case that inputs have either two \cite{yamagami20182} or three \cite{eldridge1988determination} digits, and our first main result is a resolution of Yamagami's question for $4$ digits.  In this case, a result of \cite{hasse1978determination} shows that a base-$b$ non-trivial 4-digit fixed point exists iff $b \in \{2,4\}$ or $b$ is a multiple of $5$.  To determine $M_b$, we therefore need only consider these bases, which we accomplish as follows.

\begin{theorem}\label{theorem:iteration length}
For each $b \in \{2,4\} \cup \{5, 10, 15, 20, \ldots\}$, let $M_b$ be the largest finite distance that a $4$-digit base-$b$ number is from a fixed point of the Kaprekar function.  Then we have
\begin{itemize}
\item[(i)] $M_2 = 1, M_4 = 3, M_{5} = 4$, $M_{10} = 7$, and $M_{20} = 10$.
\item[(ii)] If $b = 5m \cdot 2^n$ for some odd number $m > 1$, then $M_b = n+2$.
\item[(iii)] Finally, if $b = 5 \cdot 2^n$ and $n \geq 3$, then we have
\[ M_b = \begin{cases}
		4n + 6 & \quad n \equiv 0 \pmod 4 \\
		3n + 5 & \quad n \equiv 1 \pmod 4 \\
		3n + 5 & \quad n \equiv 2 \pmod 4 \\
		5n + 6 & \quad n \equiv 3 \pmod 4
\end{cases}
\]
\end{itemize}
\end{theorem}
The above result (and our proof) splits into several cases, the need for which is made evident when examining a plot of $M_b$, as in Figure \ref{figure:mb}.  From this, we see most values of $b$ exhibit a fractal-like pattern with several interspersed much larger values (namely those where $b = 5 \cdot 2^n$).  Some version of (ii) [governing the fractal-like behavior] can be inferred from work present in \cite{hasse1978determination}, and we include a proof of this for completeness.

The results of \cite{hasse1978determination} on the $4$-digit Kaprekar function prove that for each base divisible by $5$, there is exactly one non-trivial fixed point, the digits of which are $(3b/5)(b/5 -1)(4b/5 -1)(2b/5)$.  Since this non-trivial fixed point is unique, it was a notable matter of interest to ask when (as in the base-$10$ case) every input not divisible by $1111$ eventually reaches this fixed point.  As before, let $S_b$ denote the set of $4$-digit inputs which eventually reach this fixed point.  The main result of \cite{hasse1978determination} was that $|S_b| = b^4 - b$ iff $b = 5 \cdot 2^{n}$ for $n=0$ or $n$ odd (the $-b$ term accounting for the $b$ multiples of $1111$).  Our next result generalizes this to study $C_b = |S_b|/b^4$---the fraction of all input which eventually reach the non-trivial fixed point---for other bases $b$.

\begin{theorem}\label{theorem:proportion}
If $b = 5m \cdot 2^n$, for some odd $m > 1$, then $\displaystyle C_b  = \dfrac{|S_b|}{b^4} = \dfrac{8 + 40 \cdot 4^n}{5 \cdot b^2} = \dfrac{8}{5 \cdot b^2} + \dfrac{8}{25 m^2}.$
\end{theorem}

\begin{figure}[ht]
	\begin{center}
		\includegraphics[height=1.5in]{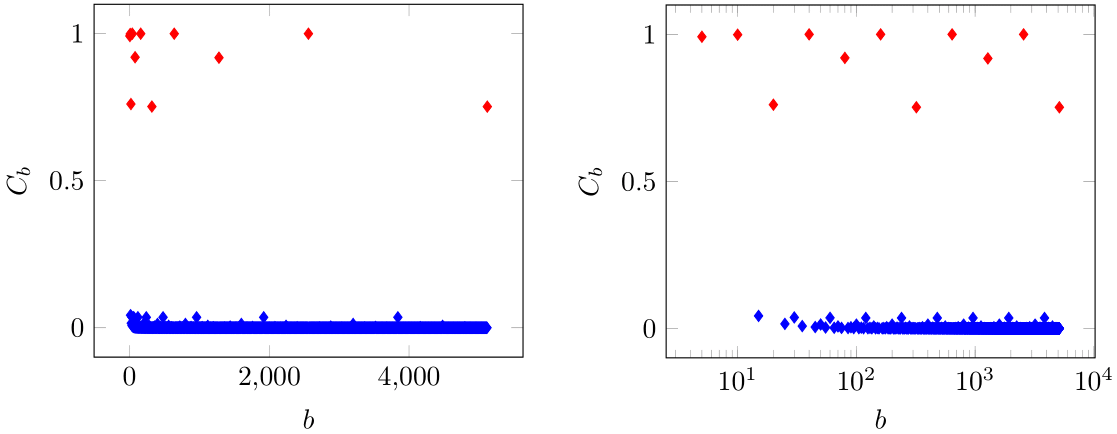}
		\includegraphics[height=1.5in]{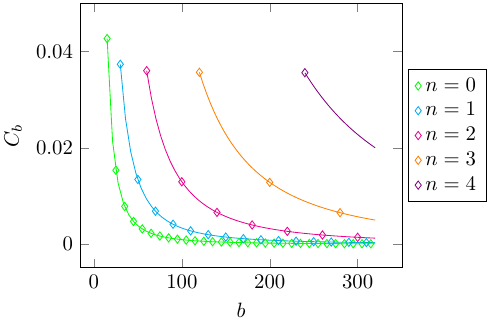}
	\end{center}
	\caption{Plots of $C_b$. Points colored according to values of $n$.}\label{figure:cb}
\end{figure}
In terms of Figure \ref{figure:cb}, the above theorem describes the behavior of all the points where $C_b$ is relatively small (i.e., those points clustered near $C_b =0$ in the first two graphs).  The last graph in the figure zooms in on this region and overlays our proven formula for these values.  As discussed, \cite{hasse1978determination} provides a formula for $C_b$ when $b = 5 \cdot 2^{n}$ and $n$ odd.  For even $n$, the middle graph in Figure \ref{figure:cb} seems to suggest there is some fractal-like self-similarity.  Numerically, it seems like these values of $C_b$ converge, and a somewhat bold conjecture might be that $C_b \to 3/4$ for $b = 5\cdot 2^{4n}$ and that $C_b \to 15/16$ for $b = 5\cdot 2^{4n+2}$.  We are unsure if either of these are true (or how a proof might go), but we believe a resolution either way would be interesting.

\subsection*{Outline of the paper}
We begin in Section \ref{section:difference pairs} by introducing a useful transformation originally due to \cite{hasse1978determination} that we will be using throughout.  In Section \ref{section:irregular bases}, we study bases of the form $b =5m \cdot 2^n$ with $m > 1$ odd.  In this same section, we prove Theorem \ref{theorem:iteration length}(ii) as well as Theorem \ref{theorem:proportion}.  In Section \ref{section:regular bases}, we address bases of the form $5 \cdot 2^n$ and prove Theorem \ref{theorem:iteration length}(iii).  We conclude in Section \ref{section:conclusion} with a discussion of open problems.

\subsection*{Acknowledgments}
This work was done as a part of the Summer Undergraduate Math Research at Yale (SUMRY) program.

The second author is grateful for funding from Yale's First-year Summer Research Fellowship.  He would also like to thank Vinton Geistfeld for inspiring his interest in this problem many years ago. 

\section{Difference pairs}\label{section:difference pairs}
We begin with the following transformation, which is very useful in studying the $4$-digit Kaprekar function.

If the four base-$b$ digits of $x$ are $a_3 \geq a_2 \geq a_1 \geq a_0$, then its (base-$b$) \emph{\textbf{difference pair}} is given by $(a_3 - a_0, a_2 - a_1)$. For a given difference pair $X = (d, d')$, we will write $cX = c(d, d')$ to mean $(cd, cd')$.  With this notation, Hasse and Prichett proved that the fixed point when $5 | b$ has difference pair $(3b/5,\ b/5)$.

There is an important distinction to be made between a fixed point of $K$ and a fixed pair. It is possible for a number to have the same difference pair as a fixed point of $K$ and not be a fixed point itself. (e.g. 6174 and 8532 have the same difference pair).  That said, Hasse and Prichett \cite{hasse1978determination} made the observation that $K(x)$ is completely determined by the difference pair of $x$, allowing us to extend the definition of Kaprekar's function to difference pairs.  But first, we classify difference pairs into three types:

\medskip

\noindent \textbf{Definition:} \emph{Let $(d, d')$ be a difference pair in base $b$.\\
	\indent If $d > d'$ and $d + d' \neq b$, we say that $(d, d')$ is \textbf{type (a)}. \\
	\indent If $d = d'$ or $d + d' = b$, we say that $(d, d')$ is \textbf{type (b)}. \\
	\indent If $d > d'$ and $d' = 0$, we say that $(d, d')$ is \textbf{type (c)}.}

\medskip

\noindent With this, we can define $K$ for difference pairs as follows.  $K((d, d')) = (d_1, d_1')$, where
	\[ \{d_1, d_1 ' \} = \begin{cases}
		\big\{ \vert 2d - b\vert, \vert 2d' - b\vert \big\} & \quad (d,d') \text{ is type (a)} \\
		\big\{ \vert 2d - (b - 1)\vert, \vert 2d - (b + 1)\vert\big\} & \quad (d,d') \text{ is type (b)} \\
		\big\{d - 1, b - d\big\} & \quad (d,d') \text{ is type (c)} \\
		0 & \quad d = d' = 0
	\end{cases} \]

Importantly, under this definition, if $x$ has difference pair $(d, d')$ and $y = K(x)$ has difference pair $(d_1, d_1')$. Then $K((d_0, d_1)) = (d_0', d_1')$.  Moreover, if $x$ and $z$ have the same difference pair, then $K(x)$ and $K(z)$ do as well [in fact, we'd also have $K(x) = K(z)$ as integers].

It turns out to be quite useful to be able to list out all possible predecessors of a particular difference pair. Hasse and Prichett \cite{hasse1978determination} provided a table describing just that, and below we provide two tables.

\vspace{5pt}

\begin{center}
	\begin{tabular}{||c|c|c|c|c||}
		\hline
		Type & Predecessor Type & \multicolumn{2}{c|}{Predecessors} & Conditions \\
		\hline 
		\multirow{5}{*}{(a)} & \multirow{2}{*}{(a)} & \multicolumn{1}{c}{$(\frac{b + d}{2}, \frac{b + d'}{2})$} & \multicolumn{1}{c|}{$(\frac{b + d}{2}, \frac{b - d'}{2})$} & $d\equiv d' \equiv b \pmod 2$ \\
		& & \multicolumn{1}{c}{$(\frac{b - d'}{2}, \frac{b - d}{2})$} & \multicolumn{1}{c|}{$(\frac{b + d'}{2}, \frac{b - d}{2})$} & \\
		\cline{2-5}
		& \multirow{2}{*}{(b)} & \multicolumn{1}{c}{$(\frac{b - 1 + d}{2}, \frac{b - 1 + d}{2})$} & \multicolumn{1}{c|}{$(\frac{b - 1 + d}{2}, \frac{b + 1 - d}{2})$} & $d\equiv d' \equiv b + 1 \pmod 2$ \\
		& & \multicolumn{1}{c}{$(\frac{b + 1 - d}{2}, \frac{b + 1 - d}{2})$} & \multicolumn{1}{c|}{} & $d = d' + 2$ \\
		\cline{2-5}
		& (c) & \multicolumn{1}{c}{$(d + 1, 0)$} & \multicolumn{1}{c|}{$(d' + 1, 0)$} & $d + d' = b - 1$ \\
		\hline 
		\multirow{6}{*}{(b)} & \multirow{2}{*}{(a)} & \multicolumn{1}{c}{$(\frac{b + d}{2}, \frac{b + d'}{2})$} & \multicolumn{1}{c|}{$(\frac{b + d}{2}, \frac{b - d'}{2})$} & $d + d' = b,\ \  d \neq d'$ \\
		& & \multicolumn{1}{c}{$(\frac{b - d'}{2}, \frac{b - d}{2})$} & \multicolumn{1}{c|}{$(\frac{b + d'}{2}, \frac{b - d}{2})$} & $d \equiv d' \equiv 0 \pmod 2$ \\
		\cline{2-5}
		& \multirow{2}{*}{(b)} & \multicolumn{1}{c}{$(\frac{3b}{4}, \frac{3b}{4})$} & \multicolumn{1}{c|}{$(\frac{3b}{4}, \frac{b}{4})$} & $d = d' + 2,\ \  b \equiv 0 \pmod 4$ \\
		& & \multicolumn{1}{c}{$(\frac{b}{4}, \frac{b}{4})$} & & \\
		\cline{3-5}
		& & \multicolumn{1}{c}{$(\frac{b}{2}, \frac{b}{2})$} & \multicolumn{1}{c|}{} & $d=d' = 1,\ \ b \equiv 0 \pmod 2$ \\
		\cline{2-5} 
		& (c) & \multicolumn{1}{c}{$(\frac{b + 2}{2}, 0)$} & \multicolumn{1}{c|}{} & $d = d' = \frac{b + 1}{2},\ \ b \equiv 1 \pmod 2$ \\
		\hline 
		\multirow{4}{*}{(c)} & (a) & \multicolumn{1}{c}{$(\frac{b + d}{2}, \frac{b}{2})$} & \multicolumn{1}{c|}{$(\frac{b}{2}, \frac{b - d}{2})$} & $b \equiv d \equiv 0 \pmod 2$ \\
		\cline{2-5}
		& \multirow{2}{*}{(b)} & \multicolumn{1}{c}{$(\frac{b + 1}{2}, \frac{b + 1}{2})$} & \multicolumn{1}{c|}{$(\frac{b + 1}{2}, \frac{b - 1}{2})$} & $d = 2,\ \  b \equiv 1 \pmod 2$ \\
		& & \multicolumn{1}{c}{$(\frac{b - 1}{2}, \frac{b - 1}{2})$} & & \\
		\cline{2-5}
		& (c) & \multicolumn{1}{c}{$(1, 0)$} & & $d = b - 1$ \\
		\hline 
	\end{tabular}
\end{center}

Alternatively, we can condense the contents of the table even further as follows:

\begin{framed}
\noindent Let $4 < b$ be any base divisible by four.  The comprehensive list of what precedes each difference pair is:
\begin{itemize}
\item $(0,0) \leftarrow (0,0)$ and $(1,1) \leftarrow (b/2, b/2)$ and also $(b-1,0) \leftarrow (1,0)$
\item[(i)] for $i \neq j$, we have $(2i, 2j) \leftarrow \left(\frac{b}{2} \pm i, \frac{b}{2} \pm j \right)$
\item[(ii)] for all $k$, we have $(2k+1, 2k-1) \leftarrow \left(\frac{b}{2} \pm k, \frac{b}{2} \pm k \right)$
\item[(iii)] For $k \notin \{0, b-1\}$, there are two predecessors of $(k, b-1-k)$.  Namely $(k, b-1-k) \leftarrow (k+1,0)$ and also $(k, b-1-k) \leftarrow (b-k,0)$
\item Pairs not listed above have no predecessors [in particular $(x,x)$ has predecessors iff $x \in \{0,1\}$]
\end{itemize}
\end{framed}

We now explore the contrast of applying $K$ to an integer as opposed to applying it to its difference pair.
\begin{lemma}\label{lemma:fixed point pair}
	For any $b \equiv 0 \pmod 5$, if $x$ has difference pair $(\frac{3b}{5}, \frac{b}{5})$, then $K(x)$ is the fixed point in base $b$. Furthermore, if $(d, d')$ is a first generation predecessor of the fixed pair, and $y$ has difference pair $(d, d')$, then $K(y) \neq K(x)$. 
\end{lemma}
\begin{proof}
	Suppose $x$ has difference pair $(\frac{3b}{5}, \frac{b}{5})$ and has the digits $\overline{ \left( \frac{3b}{5} + c \right) \left( \frac{b}{5} + d \right) \left(d\right) \left(c\right) }$.  Then $K(x)$ is 
	\begin{center}
		\begin{tabular}{ccccc}
			& $(\frac{3b}{5} + c)$ & $(\frac{b}{5} + d)$ & $(d)$ & $(c)$ \vspace*{.1in}\\
			$-$ & $(c)$ & $(d)$ & $(\frac{b}{5} + d)$ & $(\frac{3b}{5} + c)$ \vspace*{.1in}\\
			\hline
			\vspace*{.1in}
			& $(\frac{3b}{5})$ & $(\frac{b}{5})$ & $(\frac{-b}{5})$ & $(\frac{-3b}{5})$ \\
			$=$ & $(\frac{3b}{5})$ & $(\frac{b}{5} - 1)$ & $(\frac{4b}{5} - 1)$ & $(\frac{2b}{5})$
		\end{tabular}
	\end{center}
	The second statement is follows by a similar subtraction argument for each of the immediate difference-pair predecessors of the fixed point: $\left(\frac{4b}{5}, \frac{3b}{5}\right), \left(\frac{4b}{5}, \frac{2b}{5}\right),$ and $\left(\frac{2b}{5}, \frac{b}{5}\right)$.
\end{proof}

\section{Bases of the form $b = 5m \cdot 2^n$, for odd $m >1$}\label{section:irregular bases}
We first turn our attention to bases of the form $5m \cdot 2^n$, with $m > 1$ odd.  For such bases, the behavior of $K$ is relatively easy to understand.  This was already noted as early as \cite{hasse1978determination}, and in fact their argument---much like ours---was split into these exceptional bases and the comparatively more regular bases discussed in later sections.  We first provide a proof of Theorem \ref{theorem:iteration length}(ii), essentially following the treatment in \cite{hasse1978determination}.
\subsection{Distance to the fixed point, $M_b$}
\begin{proof}[Proof of Theorem \ref{theorem:iteration length}(ii)]
Letting $b = m \cdot 5 \cdot 2^n$, we proceed by induction on $n$.  In particular, we will prove the claim that (i) every difference-pair predecessor of the fixed point is of type (a); (ii) each coordinate of every difference-pair predecessor is divisible by $m$; and (iii) and $M_b = n+2$.

\paragraph*{Base case:} Suppose $n = 0$. In general, we know the fixed point has as immediate predecessors 
\[ \left(\frac{3b}{5}, \frac{b}{5}\right), \left(\frac{4b}{5}, \frac{3b}{5}\right), \left(\frac{4b}{5}, \frac{2b}{5}\right), \left(\frac{2b}{5}, \frac{b}{5}\right).
\]
Because $b = 5m$, these can be rewritten as $(3m, m), (4m, 3m), (4m, 2m), (2m, m).$
The first of these is the fixed point, so we need not consider it any further. As for the rest, $m$ is odd, so these all have at least one even component, which preclude them from having any type (a) predecessors [because $b$ is odd]. The difference between their components is either $m$ or $2m$, and since $m$ is odd and not equal to one, we never have $d = d' + 2$, which is a requirement to have type (b) predecessors. Finally, the sum of their components is either $3m$, $6m$, or $7m$, which cannot equal $b + 1 = 5m + 1$, so they have no type (c) predecessors either. Then by Lemma \ref{lemma:fixed point pair}, base $b$ has $M_b = 2$.

\paragraph*{Induction step:} Suppose the desired claim holds for some integer $n \geq 0$.  Let $(x, y)$ be a highest generation predecessor of $X$ (i.e., $(x, y)$ has no predecessors of its own).  In base $2b$, we know that the fixed point is $2X$ and has predecessors $2p$ for every predecessor $p$ of $X$ in base $b$.  Moreover, it also has the immediate predecessors of these pairs.  Namely, considering the predecessors of $(2x, 2y)$ [in base $2b$], we see these are 
	\[ (b + x, b + y), (b + x, b - y), (b - y, b - x), (b + y, b - x) \]
For each of these predecessors, we note that $d$ and $d'$ are both divisible by $m$, and they are not equal.  This means they must be of type (a) since their difference cannot be less than $2$ and their sum cannot be $b-1$ (since neither of these quantities are multiples of the odd number $m\geq 3$).  Moreover, if $(x,y)$ had no base-$b$ predecessors, then the four pairs listed above cannot have any base-$2b$ predecessors since we must have $x \equiv 1+y \pmod{2}$.  Thus, we see that $M_{2b} = 1 + M_{b}$, completing the proof.
\end{proof}

\subsection{Convergence rate}
Our proof of Theorem \ref{theorem:proportion} uses the notion of difference pairs, but the statement is about the number of inputs in $S_b$.  Thus, we need the following, which relates how many four-digit numbers have a given difference pair.

\begin{lemma}\label{lemma:number for each difference pair}
For a difference pair $(d, d')$ with $d \neq d'$ there are $N((d, d')) := 24(b- d)(d - d')$ four-digit base-$b$ numbers with difference pair $(d, d')$.
\end{lemma}

\begin{proof}
	Suppose $x$, which has type (a) difference pair $(d, d')$, has digits $a_3 \geq a_2 \geq a_1 \geq a_0$. Notice that $\{a_3, a_0\}$ can take values from $\{d, 0\}, \cdots, \{b - 1, b - d - 1\}$, $b - d$ possibilities. Similarly, $\{a_2, a_1\}$ can take values from $\{a_3, a_3 - d'\}, \cdots, \{a_0 + d', a_0\}$, $d - d' + 1$ possibilities. Since $x$ is type (a), $d' \neq 0$ and $d' \neq d$. This means that the only situation where not all digits are distinct is when either $a_2 = a_3$ and $a_1 \neq a_0$ or $a_1 = a_0$ and $a_2 \neq a_3$. In both of these cases, $x$ has exactly one pair of duplicate digits. Then we can compute how many $x$ have type (a) difference pair $(d, d')$ to be 
	\[ 4!(b - d)(d - d' - 1) + \frac{4!}{2}(b - d)(2) = 24(b - d)(d - d' - 1) + 24(b - d) = 24(b - d)(d - d') \]
\end{proof}

Now we are equipped to prove Theorem \ref{theorem:proportion}.

\begin{proof}[Proof of Theorem \ref{theorem:proportion}]
We let $m > 1$ be fixed and proceed by induction on $n$.  In particular, letting $A_n = C_b b^4$ we will prove the claim that (i) $A_n = 40 \cdot 4^n \cdot m^2(1 + 5 \cdot 4^n)$; (ii) the fixed point has $4^{n+1}$ difference-pair predecessors (including itself) all of which are type (a); and (iii) these difference-pair predecessors can be partitioned into sets of the form $H = \{(x,y), (x, b-y), (y, b-x), (b-y, b-x)\}$.

\paragraph*{Base case ($n=0$):}	In general, we know that the fixed point has only four immediate difference-pair predecessors (including itself), all of type (a): 
\[
\left(\frac{3b}{5}, \frac{b}{5} \right), \left(\frac{4b}{5}, \frac{3b}{5}\right), \left(\frac{4b}{5}, \frac{2b}{5}\right), \left(\frac{2b}{5}, \frac{b}{5}\right).
\]
Moreover, since $m$ (and $b$) are odd, all of these are type (a), and none of these has any additional predecessors.

Finally, we have
\begin{eqnarray*}
N\left(\frac{4b}{5}, \frac{3b}{5} \right) &=& 24\left(b - \frac{4b}{5}\right)\left(\frac{4b}{5} - \frac{3b}{5}\right) = 24(2^n \cdot m)(2^n \cdot m)\\
N\left(\frac{4b}{5}, \frac{2b}{5} \right) &=& 24\left(b - \frac{4b}{5}\right)\left(\frac{4b}{5} - \frac{2b}{5}\right) = 24(2^n \cdot m)(2 \cdot 2^n \cdot m)\\
N\left(\frac{2b}{5}, \frac{b}{5} \right) &=& 24\left(b - \frac{2b}{5}\right)\left(\frac{2b}{5} - \frac{b}{5}\right) = 24(3 \cdot 2^n m)(2^n \cdot m)\\
N\left(\frac{3b}{5}, \frac{b}{5} \right) &=& 24\left(b - \frac{3b}{5}\right)\left(\frac{3b}{5} - \frac{b}{5}\right) = 24(2 \cdot 2^n \cdot m)(2 \cdot 2^n \cdot m),
\end{eqnarray*}
and summing gives us $A_0 = 240 m^2$ as desired.

\paragraph*{Induction step:} Suppose our claim holds for some $n\geq 0$.  First notice that
\[
N(x, y) = 24(b - x)(x - y), \quad N(x, b - y) = 24(b - x)(x + y- b)
\]
	\[ N(b - y, b - x) = 24(y)(x - y), \quad N(y, b - x) = 24(b - y)(x + y - b). \]
Summing these four together gives us 
\[
N(H) := \sum_{(p,q) \in H}N(p,q) = 24(4bx + 2by - 2b^2 - 2x^2 - 2y^2)
\]
For each such set $H$ in base $b$, there is a corresponding set $H'$ in base $2b$ given by 
	\[ H' = \{(2x, 2y), (2x, 2b - 2y), (2b - 2y, 2b - 2x), (2y, 2b - 2x)\}. \]
These sets $H'$ will be difference-pair predecessors of the fixed point in base $2b$.  Moreover, the immediate predecessors of every element of all sets $H'$ accounts for all predecessors of the fixed point in base $2b$.  This immediately gives us parts (ii) and (iii) of our desired induction claim, and we need only prove (i).  For this, notice that $A_{n + 1} = \sum_{H'} N(K^{-1}(H'))$ across all $H'$. For every element $h$ of $H'$, we compute $N(K^{-1}(h))$, keeping in mind that we are now working in base $2b$. 
	\begin{center}
		\begin{tabular}{||c|c|c|c||}
			\hline 
			$h$ & \multicolumn{2}{c|}{Predecessors} & $N(K^{-1}(h))$ \\
			\hline 
			\multirow{2}{*}{$(2x, 2y)$} & \multicolumn{1}{c}{$(b + x, b + y)$} & \multicolumn{1}{c|}{$(b + x, b - y)$} & \multirow{2}{*}{$24(4bx - 2x^2 - 2y^2)$} \\
			& \multicolumn{1}{c}{$(b - y, b - x)$} & \multicolumn{1}{c|}{$(b + y, b - x)$} & \\
			\hline 
			\multirow{2}{*}{$(2x, 2b - 2y)$} & \multicolumn{1}{c}{$(b + x, 2b - y)$} & \multicolumn{1}{c|}{$(b + x, y)$} & \multirow{2}{*}{$24(4bx + 4by - 2b^2 - 2x^2 - 2y^2)$} \\
			& \multicolumn{1}{c}{$(y, b - x)$} & \multicolumn{1}{c|}{$(2b - y, b - x)$} & \\
			\hline 
			\multirow{2}{*}{$(2b - 2y, 2b - 2x)$} & \multicolumn{1}{c}{$(2b - y, 2b - x)$} & \multicolumn{1}{c|}{$(2b - y, x)$} & \multirow{2}{*}{$24(4bx - 2x^2 - 2y^2)$} \\
			& \multicolumn{1}{c}{$(x, y)$} & \multicolumn{1}{c|}{$(2b - x, y)$} & \\
			\hline 
			\multirow{2}{*}{$(2y, 2b - 2x)$} & \multicolumn{1}{c}{$(b + y, 2b - x)$} & \multicolumn{1}{c|}{$(b + y, x)$} & \multirow{2}{*}{$24(4bx + 4by - 2b^2 - 2x^2 - 2y^2)$} \\
			& \multicolumn{1}{c}{$(x, b - y)$} & \multicolumn{1}{c|}{$(2b - x, b - y)$} & \\
			\hline 
		\end{tabular}
	\end{center}

	We sum to obtain $N(K^{-1}(H')) = 24(16bx + 8by - 4b^2 - 8x^2 - 8y^2)$.  Thus, for any given $H$ and its corresponding $H'$, we have $N(K^{-1}(H')) = 4N(H) + 24 \cdot 4b^2$.  This then gives us
\begin{eqnarray*}
A_{n + 1} &=& \sum_{H'} N(K^{-1}(H')) = \sum_{H} \left[ 4N(H) + 24 \cdot 4b^2 \right]\\
&=& 4\sum_{H} N(H) + 24 \sum_{H} 4b^2 = 4A_n + 24 \cdot 4^{n+1} b^2 = 40 \cdot 4^{n + 1} \cdot m^2 \left(1 + 5 \cdot 4^{n + 1} \right),
\end{eqnarray*}
which completes the proof.
\end{proof}

\section{Bases of the form $b = 5\cdot 2^n$}\label{section:regular bases}
For bases of the form $5 \cdot 2^n$, a key observation of \cite{hasse1978determination} was that repeated applying $K$ to difference pairs will necessarily eventually result in a difference pair where both coordinates are divisible by $2^n = b/5$.  After this, we need only consider trajectories of difference pairs of this type.  For this, see Tables \ref{table:digraph of special points} and \ref{table:special points to ending times} of the Appendix, both of which have versions appearing in \cite{hasse1978determination}.

Thus, to understand the behavior of $K$, we need only understand how many iterations of $K$ are required before reaching a difference pair where both coordinates are divisible by $2^n$.  This depends on which difference pair is ultimately reached, as discussed as follows.

\begin{proposition}\label{proposition:predecessor length}
Let $b = 5 \cdot 2^{n}$ with $n \geq 2$.  Suppose $(u,v)$ is any fixed difference pair, and let $L \geq 0$ be the least integer for which $K^L (u,v)$ is of the form $(p 2^{n}, q 2^{n})$, where $0 \leq q \leq p < 5$.  Then we have
\[
L \leq \begin{cases}
0 \qquad &\text{if $(p,q) = (3,1)$ or $p=q$}\\
n \qquad &\text{if $(p,q) \in \Big \{ (4,1), (3,0), (4,0) \Big \}$}\\
2n \qquad &\text{if $(p,q) \in \Big \{(4,2), (2,0) \Big\}$}\\
2n+2 \qquad &\text{if $(p,q) \in \Big \{(1,0), (2,1), (3,2), (4,3) \Big\}$}.
\end{cases}
\]
Moreover, if $n \geq 5$, then for each $(p,q)$ there is a pair $(u,v)$ for which the above upper bound on $L$ is attained.
\end{proposition}

We temporarily postpone a proof of the upper bound in the above proposition, but we show the equality case simply by explicitly constructing a pair $(u,v)$ for each pair $(p,q)$.  Note that we need not consider the cases $p=q$ or $(p,q) = (3,1)$.  This is summarized in Table \ref{table:starting points} of the Appendix, and each row is easily verified.

\subsection*{Derivation of Theorem \ref{theorem:iteration length}(iii) from Proposition \ref{proposition:predecessor length}}
Table \ref{table:digraph of special points} provides all the information required to know how many iterations are needed to reach the fixed point (or to loop) provided the initial input has a difference pair of the form $2^n (p,q)$.  Moreover, Proposition \ref{proposition:predecessor length} provides the exact bound many steps are possible before reaching any input whose difference sequence is of the form $2^n (p,q)$.

Together, we can combine these two pieces of information into Tables \ref{table:special points to ending times} and \ref{table:final table numbers} of the Appendix.  For each, we are determining exactly how many iterations of $K$ are needed until arriving at the four digit fixed point [not merely a point whose difference pair is $(3b/5, b/5)$].  Thus, in order to prove Theorem \ref{theorem:iteration length}(iii) for $n \geq 5$, we need only find the largest value in each given column of Table \ref{table:starting points}.  Bases $5 \cdot 2^n$ for $n \leq 4$ are each proven by an easy exhaustive computer search.

\subsection*{Lemmas to prove Proposition \ref{proposition:predecessor length}}
\begin{lemma}\label{lemma:type a predecessors}
	Suppose $4 < b$ is divisible by $2^n$ and $0 \leq t \leq n \geq 2$.  Further suppose that $K^{t} (u, v) = (2^t c, 2^t d)$.  Then one of the coordinates of $(u,v)$ is of the form $b i / 2^t \pm c$ and the other of the form $b j / 2^t \pm d$, where $i$ and $j$ are positive odd integers each at most $2^t$.
\end{lemma}

\begin{proof}
	We'll prove this by induction on $t$.  If $t=0$, there is nothing to prove.  Suppose now that $t \geq 1$.  We know from the above table that $K(u', v') = (2^t c, 2^t d)$ implies one coordinate of $(u', v')$ is of the form $b/2 \pm 2^{t-1} c$ and the other of the form $b/2 \pm 2^{t-1} d$.  Since $t \leq n$, both of these are divisible by $2^{t-1}$, so we may apply the induction step to assert that since $K^{t-1} (u,v) = (u', v')$, the coordinates of $(u,v)$ must be of the form
	\begin{eqnarray*}
	\dfrac{b i}{2^{t-1}} \pm \left(\dfrac{b}{2^{t}} \pm c \right) &=& \dfrac{b (2i \pm 1)}{2^t} \pm c, \qquad \text{and}\\
	\dfrac{b j}{2^{t-1}} \pm \left(\dfrac{b}{2^{t}} \pm d \right) &=& \dfrac{b (2j \pm 1)}{2^t} \pm d.
	\end{eqnarray*}
	And since $1 \leq i \leq 2^{t-1}$ is odd, we have that $2i \pm 1$ is odd as well and $1 \leq 2i \pm 1 \leq 2^t$, as desired.
\end{proof}

\begin{lemma}\label{lemma:above type c predecessors}
	Suppose $b = 5 \cdot 2^{n}$ for $n \geq 2$.  Suppose $(x,0)$ is a difference pair such that $2^n$ does not divide $x$.  If $K^{L} (u, v) = (x,0)$, then $L \leq n+1$.
\end{lemma}

\begin{proof}
First assume that $L \geq n+1$ (otherwise, there is nothing to prove), and for each $0 \leq t \leq n+1$ define $K^{L-t} (u,v) = (x_t, y_t)$ [so that $K^{t} (x_t, y_t) = (x,0) = (x_0, y_0)$].

Write $x = 2^m c$ for some integer $m < n$ and some odd number $c \geq 1$.  If $m =0$, then $(x,0)$ has a predecessor only if $x = b-1$. This implies $(x_1, y_1) = (1,0)$, which has no predecessors.  Thus $L \leq 1 < n+1$.

Now assume $m \geq 1$.  Using Lemma \ref{lemma:type a predecessors}, we know that $(x_m, y_m)$ must be of the form $(bi / 2^m \pm c, bj / 2^m)$ or $(bj / 2^m, bi / 2^m \pm c)$, where $i$ and $j$ are positive odd integers at most $2^m$.  Since $i$ and $j$ are odd and since $m < n$, we have that one of the coordinates of $(x_m, y_m)$ is even and the other is odd.  Therefore, since (by assumption) $m < n < L$, we know that $(x_m, y_m)$ has a predecessor $(x_{m+1}, y_{m+1})$, which in turn has another predecessor $(x_{m+2}, y_{m+2})$.  But since $x_m \not \equiv y_m \pmod{2}$ [and $(x_m, y_m)$ has a predecessor], we see from Theorem \ref{theorem:predecessor list} that $x_m + y_m = b-1$ and therefore $(x_{m+1}, y_{m+1})$ is of the form either $(bj / 2^m +1, 0)$ or $(b-bj/2^m, 0)$.

\paragraph*{Case 1:} Consider the case $(x_{m+1}, y_{m+1}) = (bj/2^m + 1, 0)$.  Since this has a predecessor,  we need $bj/2^m + 1 = b-1$ and $(x_{m+2}, y_{m+2}) = (1,0)$.  And since $(1,0)$ has no predecessors, this implies $L \leq m+2 \leq n +1$.

\paragraph*{Case 2:} Now consider $(x_{m+1}, y_{m+1}) = (b - bj/2^m, 0) = (5 \cdot 2^{n-m} (2^{m} - j), 0)$.  By applying Lemma \ref{lemma:type a predecessors}, since $K^{n-m} (x_{n+1}, y_{n+1}) = (x_{m+1}, y_{m+1})$, we have that the coordinates of $(x_{n+1}, y_{n+1})$ are of the form
\[
b i' /2^{n-m} \pm 5 \cdot (2^{m} -j), \qquad \text{and} \qquad b j' /2^{n-m}.
\]
As before, since $0 < m < n$, one of these coordinates is even and the other is odd.  Thus, $(x_{n+1}, y_{n+1})$ has a predecessor only if $x_{n+1} + y_{n+1} = b-1$, but this is not possible since $x_{n+1} \equiv y_{n+1} \equiv b \equiv 0 \pmod{5}$.  Thus, $(x_{n+1}, y_{n+1})$ cannot have any predecessors, which implies $L \leq n+1$.
\end{proof} 

\begin{lemma}\label{lemma:above type b predecessors}
Suppose $b = 5 \cdot 2^{n}$ for $n \geq 2$, and suppose $(x,b-x)$ is a difference pair such that $2^n$ does not divide $x$.  If $K^{L} (u, v) = (x,b-x)$, then $L \leq n-1$.
\end{lemma}

\begin{proof} 
As in the previous proof, assume $L \geq n-1$, and for each $0 \leq t \leq n-1$ define $K^{L-t} (u,v) = (x_t, y_t)$.  Also, write $x = 2^m c$ for $m < n$ and $c$ odd.  By Lemma \ref{lemma:type a predecessors}, $(x_m, y_m)$ must be of the form $(b i /2^m + \varepsilon_1 c, b (j+ \varepsilon_2) /2^m - \varepsilon_2 c)$  or $(b (j+ \varepsilon_2) /2^m - \varepsilon_2 c, b i /2^m + \varepsilon_1 c)$, where $\varepsilon_1, \varepsilon_2 \in \{-1,1\}$ and $i,j$ are positive odd integers each less than $2^m$.  Therefore, $x_m \equiv y_m \equiv 1 \pmod{2}$, so if $(x_{m+1}, y_{m+1})$ exists, then we'd need either $(x_m , y_m) = (1,1)$ or $x_m - y_m = 2$.

\paragraph*{Case 1:} Suppose $(x_m, y_m) = (1,1)$.  But for $1 \leq t \leq n+1$, we have $K^{t} (1,1) = (b- 2^{t-1}, b-3 \cdot 2^{t-1})$, and none of these are of the form $(x, b-x)$.  So it's not possible to have $(x_m, y_m) = (1,1)$ since $m < n$.

\paragraph*{Case 2:} Now suppose $L \geq m+1$ and $x_m - y_m = 2$.  Then we'd need
\[
\dfrac{b}{2^m} (i-j - \varepsilon_2) + c (\varepsilon_1 + \varepsilon_2) = \pm 2.
\]
Looking at this mod $5$, we see that if $L \geq m+1$, then $c (\varepsilon_1 + \varepsilon_2) \equiv \pm 2 \pmod{5}$, and thus $\varepsilon_1 = \varepsilon_2$.  Therefore, we'd have $(x_{m+1}, y_{m+1})$ is of the form
\[
\left( \dfrac{b}{2} \pm \dfrac{b}{2^{m+2}} (i+j + \varepsilon_2) ,  \dfrac{b}{2} \pm \dfrac{b}{2^{m+2}} (i+j + \varepsilon_2) \right).
\]
Thus, $(x_{m+1}, y_{m+1}) = (5 \cdot 2^{n-m-2} p, 5 \cdot 2^{n-m-2} q)$, for odd integers $p,q$.  By Lemma \ref{lemma:type a predecessors}, the coordinates of $(x_{n-1}, y_{n-1})$ must therefore be of the form
	\[
	b i' /2^{n-m-2} \pm 5 p, \qquad \text{and} \qquad b j' /2^{n-m-2}  \pm 5q.
	\]
	And since both of these coordinates are odd and both are divisible by $5$, the pair $(x_{n-1}, y_{n-1})$ cannot have any predecessors, and thus $L \leq n-1$.
\end{proof} 

\subsection*{Proof of Proposition \ref{proposition:predecessor length}}
\begin{proof}
As noted before, the equality case is shown by cases as summarized in Table \ref{table:starting points}.  We will prove this by splitting the argument into four main cases.  Note that the proofs for $L \leq n$ appear in cases III and IV.  For $0 \leq t \leq L$, define $K^{L-t} (u,v) = (x_t, y_t)$ [so that $K^{t} (x_t, y_t) = (x_0, y_0) = (p 2^n, q 2^n)$].

\paragraph*{Case I:} Suppose either $(1,1)$ or $(b-1,0)$ appears as some $(x_t, y_t)$.
\begin{itemize}
\item If $(1,1)$ is some $(x_t, y_t)$, then $(u,v) \in \Big \{ (1,1), (b/2, b/2) \Big \}$ since $(1,1)$ has only $(b/2, b/2)$ as a predecessor, and $(b/2, b/2)$ has none.  For all $2 \leq t \leq n+2$, we have $K^{t} (b/2, b/2) = (b - 2^{t-2}, b- 3 \cdot 2^{t-2})$, and since $K^{n+2}(b/2, b/2) = (4b/5, 2b/5)$ we'd need $(p,q) = (4,2)$ and $L \leq n+2$.
\item Similarly, if $(b-1, 0)$ is some $(x_t, y_t)$, then $(u,v) \in \Big \{(b-1,0), (1,0) \Big \}$.  For all $3 \leq t \leq n+2$ we have $K^{t} (1,0) = (b- 2^{t-2}, b-2^{t-1})$, which would imply $(p,q) = (4, 3)$ and $L \leq n+2.$
\end{itemize}

\noindent Thus, we may assume neither $(1,1)$ nor $(b-1,0)$ appears as any $(x_t, y_t)$.

\paragraph{Case II:} If $p=q \neq 0$, then $(p2^{n}, p 2^{n})$ has no predecessors.  Moreover, the only predecessor of $(0,0)$ is $(0,0)$, and every immediate predecessor of $(3 b/5, b/5)$ has both coordinates divisible by $b/5 = 2^n$.  Thus, we've proven the claim for $(p,q) = (3,1)$ and for $p=q$.

\hrulefill

For each of the next two scenarios, suppose $L \geq n$, and as before, we know that the coordinates of $(x_n, y_n)$ are of the form
\[
\dfrac{b}{2^{n}} i \pm p = 5i \pm p, \qquad \text{and} \qquad \dfrac{b}{2^{n}} j \pm q = 5j \pm q,
\]
for positive odd integers $i, j$ each at most $2^n$.

\paragraph*{Case III:} Suppose $p \equiv q \pmod{2}$.  The case $p \equiv q \equiv 1 \pmod{2}$ was handled in Case I, so we need only consider $p \equiv q \equiv 0 \pmod{2}$.  In this case, we see that both coordinates of $(x_n, y_n)$ are odd, so if $(x_n, y_n)$ has a predecessor, then $x_n = y_n +2$ [since we've already ruled out the possibility that $(x_n, y_n) = (1,1)$].

Thus, if $L > n$, then $x_n - y_n = 2$, implying $\pm p \pm q \equiv \pm 2 \pmod{5}.$  If $(p,q) = (4,0)$, this is impossible, so in that case we'd need $L \leq n$ as desired.  But in general, if $L > n$, then $(x_{n+1}, y_{n+1})$ must be of the form $(b/2 \pm k, b/2 \pm k)$, where $k = (x_{n} + y_{n})/4$.  And since $(x_{n+1},x_{n+1}) \neq (1,1)$, if $(x_{n+1}, y_{n+1})$ has any predecessors we'd need $(x_{n+1}, y_{n+1}) = (b/2 + k, b/2 - k)$.  By assumption, these coordinates are not both divisible by $2^{n}$, so we may apply Lemma \ref{lemma:above type b predecessors} to say that $L - (n+1) \leq n-1$, proving $L \leq 2n$.

\paragraph*{Case IV:} Finally suppose $p \equiv q+1 \pmod{2}$ so that $x_n \equiv y_n +1 \pmod{2}$.  If $L > n$, then $x_n + y_n = b-1$, which implies $5(i+j) \pm p \pm q = 5 \cdot 2^{n} -1$.  Viewing this mod 5, this has no solutions if $(p,q) \in \Big \{(4,1), (3,0) \Big \}$, so in those cases we have $L \leq n$.  In general, we'd have $(x_{n+1}, y_{n+1}) = (x_{n+1}, 0)$.  Since $x_{n+1}$ is not divisible by $2^n$, Lemma \ref{lemma:above type c predecessors} implies $L - (n+1) \leq n+1$ as desired.
\end{proof}

\section{Concluding remarks}\label{section:conclusion}
We have already discussed the open problem of determining $C_b$ for all $b$, but there are several other interesting open questions as well.  Perhaps the most obvious direction for future work would be to extend this analysis to the case of inputs with more than $4$ digits, and \cite{prichett1978terminating} might be a good starting point.

As another example, now that $M_b$ is understood, a natural next question might be to study how many points are $t$ iterations away from the fixed point.  We provide the following charts regarding this distribution for several given bases.  Doing so reveals several intriguing features of these distributions.

\begin{figure}[h]
\begin{center}
	\includegraphics[height=1.5in]{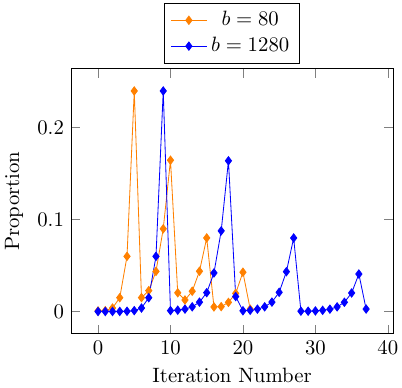}
	\includegraphics[height=1.5in]{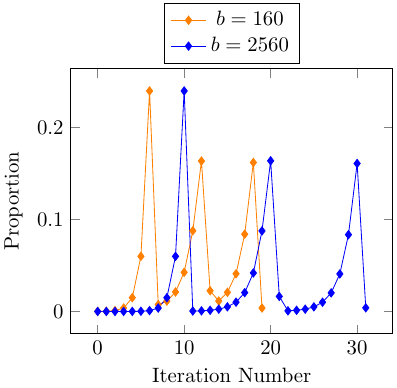}
	\includegraphics[height=1.5in]{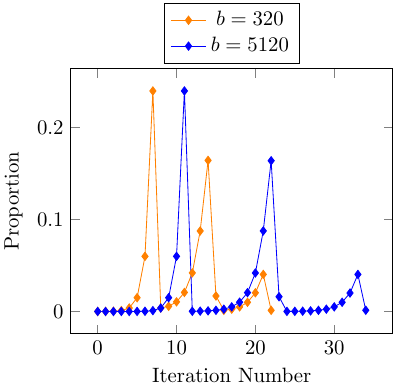}
	\includegraphics[height=1.5in]{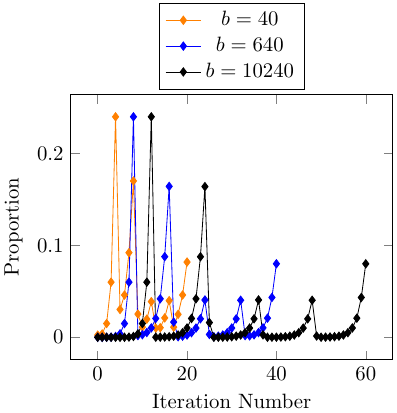}
\end{center}
\caption{Normalized distribution for how many points are $k$ iterations from the fixed point}
\end{figure}
It seems likely that the symmetry in these distributions can be explained in terms of Proposition \ref{proposition:predecessor length} (each spike consisting of roughly $n+1$ points, which grow roughly like $4^t$ until most of these predecessors die out at the same time [as we no longer have any type (a) predecessors]).  We believe this heuristic could perhaps be made rigorous, but we anticipate that getting an exact understanding would be quite delicate.
\bibliographystyle{plain}
\bibliography{references}

\begin{thebibliography}{10}

\bibitem{chaillekaprekar}
NE~Chaille.
\newblock {K}aprekar type routines for arbitrary bases.

\bibitem{dolan2011classification}
Stan Dolan.
\newblock A classification of {K}aprekar constant.
\newblock {\em The Mathematical Gazette}, 95(534):437--443, 2011.

\bibitem{eldridge1988determination}
Klaus~E Eldridge and Seok Sagong.
\newblock The determination of {K}aprekar convergence and loop convergence of
  all three-digit numbers.
\newblock {\em The American Mathematical Monthly}, 95(2):105--112, 1988.

\bibitem{sciAm}
Martin Gardner.
\newblock Mathematical games.
\newblock {\em Scientific American}, 232(3):112--117, 1975.

\bibitem{hasse1978determination}
H~Hasse and GD~Prichett.
\newblock The determination of all four-digit {K}aprekar constants.
\newblock {\em J. reine angew. Math}, 299(300):113--124, 1978.

\bibitem{kaprekar1949another}
DR~Kaprekar.
\newblock Another solitaire game.
\newblock {\em Scripta Math}, 15:244--245, 1949.

\bibitem{lapenta1979algorithm}
JF~Lapenta, AL~Ludington, and GD~Prichett.
\newblock An algorithm to determine self-producing r-digit g-adic integers.
\newblock {\em Journal fr Mathematik. Band}, 310:14, 1979.

\bibitem{ludington1979bound}
Anne~L Ludington.
\newblock A bound on {K}aprekar constants.
\newblock {\em Journal f{\"u}r die reine und angewandte Mathematik},
  310:196--203, 1979.

\bibitem{nandan2020multi}
Rishab~G Nandan and Ritvika~G Nandan.
\newblock Multi-layer encryption employing {K}aprekar routine and
  letter-proximity-based cryptograms, February~20 2020.
\newblock US Patent App. 16/600,524.

\bibitem{peterson2008kaprekar}
K~Peterson and H~Pulapaka.
\newblock The {K}aprekar routine and other digit games for undergraduate
  exploration.
\newblock {\em Journal of Mathematics and Science: Collaborative Explorations},
  10(1):143--156, 2008.

\bibitem{prichett1981determination}
GD~Prichett, AL~Ludington, and JF~Lapenta.
\newblock The determination of all decadic {K}aprekar constants.
\newblock {\em Fibonacci Quarterly}, 19(1):45--52, 1981.

\bibitem{prichett1978terminating}
Gordon~D Prichett.
\newblock Terminating cycles for iterated difference values of five digit
  integers.
\newblock {\em Journal f{\"u}r die reine und angewandte Mathematik},
  303:379--388, 1978.

\bibitem{trigg1972kaprekar}
Charles~W Trigg.
\newblock Kaprekar's routine with five-digit integers.
\newblock {\em Mathematics Magazine}, 45(3):121--129, 1972.

\bibitem{walden2005searching}
Byron~L Walden.
\newblock Searching for {K}aprekar's constants: algorithms and results.
\newblock {\em International Journal of Mathematics and Mathematical Sciences},
  2005(18):2999--3004, 2005.

\bibitem{yamagami20182}
Atsushi Yamagami.
\newblock On 2-adic {K}aprekar constants and 2-digit {K}aprekar distances.
\newblock {\em Journal of Number Theory}, 185:257--280, 2018.

\bibitem{yamagami2019some}
Atsushi Yamagami and Y{\=u}ki Matsui.
\newblock On some formulas for {K}aprekar constants.
\newblock {\em Symmetry}, 11(7):885, 2019.

\bibitem{young1993variation}
Anne~Ludington Young.
\newblock A variation on the two-digit {K}aprekar routine.
\newblock {\em Fibonacci Quarterly}, 31:138--145, 1993.

\bibitem{young1995switch}
Anne~Ludington Young.
\newblock The switch, subtract, reorder routine.
\newblock {\em Fibonacci Quarterly}, 33(5):432--440, 1995.

\end{thebibliography}

\newpage
\section*{Appendix: Tables}
\vfill
\begin{table}[h]
\begin{center}
\begin{tabular}{c||c|c|c|c}
Starting & \multicolumn{4}{c}{After exactly $n+1$ steps, we arrive at}\\
point & $n =4k$ & $n = 4k +1$ & $n = 4k+2$ & $n = 4k + 3$\\
\hline \hline
$2^n$(1,0) & $2^n$(1,0) & $2^n$(4,3) & $2^n$(2,1) & $2^n$(3,2)\\
$2^n$(2,0) & $2^n$(4,3) & $2^n$(2,1) & $2^n$(3,2) & $2^n$(1,0)\\
$2^n$(3,0) & $2^n$(3,2) & $2^n$(1,0) & $2^n$(4,3) & $2^n$(2,1)\\
$2^n$(4,0) & $2^n$(2,1) & $2^n$(3,2) & $2^n$(1,0) & $2^n$(4,3)\\
\hline
$2^n$(4,1) & $2^n$(4,2) & $2^n$(2,0) & $2^n$(4,2) & $2^n$(2,0)\\
$2^n$(1,1) & $2^n$(4,2) & $2^n$(2,0) & $2^n$(4,2) & $2^n$(2,0)\\
$2^n$(4,4) & $2^n$(4,2) & $2^n$(2,0) & $2^n$(4,2) & $2^n$(2,0)\\

\hline
$2^n$(3,2) & $2^n$(2,0) & $2^n$(4,2) & $2^n$(2,0) & $2^n$(4,2)\\
$2^n$(2,2) & $2^n$(2,0) & $2^n$(4,2) & $2^n$(2,0) & $2^n$(4,2)\\
$2^n$(3,3) & $2^n$(2,0) & $2^n$(4,2) & $2^n$(2,0) & $2^n$(4,2)\\
\end{tabular}
\caption{(Derived in \cite{hasse1978determination}) Trajectories of difference pairs where each coordinate is divisible by $2^n = b/5$.  Note the behavior depends on the residue of $n$ mod 4.  For the $5$ pairs not listed above, we have $K(0,0) = (0,0)$ and also $K(2b/5,b/5) = K(3b/5,b/5) = K(4b/5,2b/5) = K(4b/5,3b/5) = (3b/5,b/5)$.}
\label{table:digraph of special points}
\end{center}
\end{table}
\vfill
\begin{table}[h]
\begin{centering}
\begin{tabular}{c|c||c|c|c|c}
Starting & After this & \multicolumn{4}{c}{First pair with both coordinates divisible by $2^n$}\\
here & many steps & $n=4k$ & $n = 4k +1$ & $n = 4k+2$ & $n = 4k + 3$\\
\hline
\hline
(4,1) & $n$ & $2^n (4,1)$ & $2^n (4,1)$ & $2^n (4,1)$ & $2^n (4,1)$\\
(5,1) & $n$ & $2^n (4,0)$ & $2^n (4,0)$ & $2^n (4,0)$ & $2^n (4,0)$\\
(5,2) & $n$ & $2^n (3,0)$ & $2^n (3,0)$ & $2^n (3,0)$ & $2^n (3,0)$\\ \hline
(9,1) & $2n$ & $2^n (2,0)$ & $2^n (4,2)$ & $2^n (2,0)$ & $2^n (4,2)$\\
(7,3) & $2n$ & $2^n (4,2)$ & $2^n (2,0)$ & $2^n (4,2)$ & $2^n (2,0)$\\ \hline
$(b/2,\ 5)$ & $2n+2$ & $2^n$(3,2) & $2^n$(1,0) & $2^n$(4,3) & $2^n$(2,1)\\
$(b/4,\ 5)$ & $2n+2$ & $2^n$(2,1) & $2^n$(3,2) & $2^n$(1,0) & $2^n$(4,3)\\
$(b/8,\ 5)$ & $2n+2$ & $2^n$(4,3) & $2^n$(2,1) & $2^n$(3,2) & $2^n$(1,0)\\
$(b/16,\ 5)$ & $2n+2$ & $2^n$(1,0) & $2^n$(4,3) & $2^n$(2,1) & $2^n$(3,2)
\end{tabular}
\caption{Table showing the equality case in Proposition \ref{proposition:predecessor length}. In each of the last four rows, we need that the first coordinate be stricty larger than $5$, whereas for the rest of the table, we only need $n \geq 2$.}
\label{table:starting points}
\end{centering}
\end{table}
\vfill

\begin{table}
\begin{center}
\begin{tabular}{c||c|c|c|c}
Starting & \multicolumn{4}{c}{Number of steps until we reach the fixed point}\\
point & $n =4k$ & $n = 4k +1$ & $n = 4k+2$ & $n = 4k + 3$\\
\hline \hline
$2^n$(2,1) & 2 & 2 & 2 & 2\\
$2^n$(4,2) & 2 & 2 & 2 & 2\\
$2^n$(4,3) & 2 & 2 & 2 & 2\\
\hline
$2^n$(1,0) & Cycles (n/a)  & $n+3$ & $n+3$ & $2n+4$\\
$2^n$(2,0) & $n+3$ & $n+3$ & Cycles (n/a) & $3n+5$\\
$2^n$(3,0) & $2n+4$ & $2n+4$ & $n+3$ & $n+3$\\
$2^n$(4,0) & $n+3$ & $2n+4$ & $2n+4$ & $n+3$\\
\hline
$2^n$(4,1) & $n+3$ & $2n+4$ & $n+3$ & $4n+6$\\
$2^n$(1,1) & $n+3$ & $2n+4$ & $n+3$ & $4n+6$\\
$2^n$(4,4) & $n+3$ & $2n+4$ & $n+3$ & $4n+6$\\

\hline
$2^n$(3,2) & $2n+4$ & $n+3$ & Cycles (n/a) & $n+3$\\
$2^n$(2,2) & $2n+4$ & $n+3$ & Cycles (n/a) & $n+3$\\
$2^n$(3,3) & $2n+4$ & $n+3$ & Cycles (n/a) & $n+3$\\
\end{tabular}
\caption{Trajectories derived from Table \ref{table:digraph of special points}}
\label{table:special points to ending times}
\end{center}
\end{table}
\vfill

\begin{table}[h]
\begin{center}
\begin{tabular}{c||c|c|c|c}
First pair $2^n (p,q)$ & \multicolumn{4}{c}{Tight bound on total \# steps to reach fixed point}\\
 encountered & $n =4k$ & $n = 4k +1$ & $n = 4k+2$ & $n = 4k + 3$\\
\hline
\hline
$2^n$(2,1) & $2n+4$ & $2n+4$ & $2n+4$ & $2n+4$\\
$2^n$(4,2) & $2n+2$ & $2n+2$ & $2n+2$ & $2n+2$\\
$2^n$(4,3) & $2n+4$ & $2n+4$ & $2n+4$ & $2n+4$\\
\hline
$2^n$(1,0) & [$2n+3$]* & $3n+5$ & $3n+5$ & $4n+6$\\
$2^n$(2,0) & $3n+3$ & $3n+3$ & [$2n+1$]* & $5n+5$\\
$2^n$(3,0) & $3n+4$ & $3n+4$ & $2n+3$ & $2n+3$\\
$2^n$(4,0) & $2n+3$ & $3n+4$ & $3n+4$ & $2n+3$\\
\hline
$2^n$(4,1) & $2n+3$ & $3n+4$ & $2n+3$ & $5n+6$\\
$2^n$(1,1) & $n+3$ & $2n+4$ & $n+3$ & $4n+6$\\
$2^n$(4,4) & $n+3$ & $2n+4$ & $n+3$ & $4n+6$\\

\hline
$2^n$(3,2) & $4n+6$ & $3n+5$ & [$2n+3$]* & $3n+5$\\
$2^n$(2,2) & $2n+4$ & $n+3$ & [2]* & $n+3$\\
$2^n$(3,3) & $2n+4$ & $n+3$ & [2]* & $n+3$\\
\hline

$2^n$(0,0) & [$1$]* & [$1$]* & [$1$]* & [$1$]*\\
$2^n$(3,1) & $1$ & $1$ & $1$ & $1$
\end{tabular}
\caption{The tightness of these bounds requires $n \geq 5$.  Numbers written as $[a]$* indicate that this does not lead to the fixed point associated with $2^n (3,1)$ but instead enters a loop.  In this case, the number indicates the total number of steps until the process first reaches a value $x$ for which $K^{L} (x) = x$ for some $L$ (this is counting the number of steps to reach this value $x$, not merely a value having the same difference pair).}
\label{table:final table numbers}
\end{center}
\end{table}

\end{document}